\documentclass{amsart}

\usepackage{amsmath,amsthm,amssymb, graphicx}
\usepackage{times}
\usepackage{enumerate, xspace}
\usepackage{color}
\usepackage{cancel}

\usepackage{hyperref}

\usepackage[symbol]{footmisc}

\newtheorem{thm}{Theorem}[section]

\newtheorem{lem}[thm]{Lemma}

\newtheorem*{mainthm}{Dichotomy Theorem}

\theoremstyle{definition}

\newtheorem{defn}[thm]{Definition}

\newtheorem*{nota}{Notation}

\numberwithin{equation}{section}

\DeclareMathOperator{\ran}{ran}

\DeclareMathOperator{\Stab}{Stab}

\renewcommand{\phi}{\varphi}

\providecommand{\Id}{\text{Id}}

\providecommand{\restriction}{\uprightharpoon}

\providecommand{\Esetn}{{E^{n}_{set}}}
\providecommand{\Ece}{{=^{ce}}}

\providecommand{\EZce}{{E_0^{\text{ce}}}}

\providecommand{\eiti}{enumerable in the indices\xspace}
\providecommand{\ceoer}{c.e.\ orbit equivalence relation\xspace}
\providecommand{\cbers}{{cbers}\xspace}
\providecommand{\cber}{{cber}\xspace}

\title[Countable Borel equivalence relations in the setting of computable reducibility]{Analogues of the countable Borel equivalence relations in the setting of computable reducibility}

\author[U.~Andrews]{Uri Andrews}
\address{Department of Mathematics\\
University of Wisconsin\\
USA}
\email{{andrews@math.wisc.edu}}

\author[L.~San Mauro]{Luca San Mauro}
\address{Department of Philosophy\\University of
Bari\\
Italy}
\email{\href{mailto:luca.sanmauro@gmail.com}{luca.sanmauro@gmail.com}}

\begin{document}

\date{}

\maketitle

\begin{abstract}
	Coskey, Hamkins, and Miller \cite{CHM} proposed two possible analogues of the class of countable Borel equivalence relations in the setting of computable reducibility of equivalence relations on the computably enumerable (c.e.) sets. The first is based on effectivizing the Lusin--Novikov theorem while the latter is based on effectivizing the Feldman--Moore theorem. They asked for an analysis of which degrees under computable reducibility are attained under each of these notions.
	
	We investigate these two notions, in particular showing that the latter notion has a strict dichotomy theorem: Every such equivalence relation is either equivalent to the relation of equality ($=^{ce}$) or almost equality ($E_0^{ce}$) between c.e.\ sets. For the former notion, we show that this is not true, but rather there are both chains and antichains of such equivalence relations on c.e.\ sets which are between $\Ece$ and $E_0^{ce}$.
	This gives several strong answers to \cite[Question 3.5]{CHM} showing that in general there is no analogue of the Glimm-Efros dichotomy for equivalence relations on the c.e.\ sets. 
		%On the other hand, for a natural class of equivalence relations, such a dichotomy does hold.
\end{abstract}

\section{Introduction}
Invariant descriptive set theory \cite{gao} studies the complexity of equivalence relations up to Borel reducibility. Such a theory serves as a theoretical framework for investigating the complexity of classification problems naturally arising in mathematics.  A fundamental subclass of Borel equivalence relations is that of
\emph{countable Borel equivalence relations} (\emph{cbers}), i.e., those  whose equivalence classes are countable. By the  Feldman--Moore theorem \cite{feldman1977ergodic}, it turns out that \cbers are exactly the  orbit equivalence relations generated by Borel actions of countable groups and this brings into the subject deep
connections with group theory, ergodic theory, and operator algebras (see, e.g., \cite{jackson2002countable,kechris2019theory}). The Feldman--Moore theorem is a straightforward consequence of a classic uniformization result, due to Lusin and Novikov, which ensures that all \cbers  admit a
uniform Borel enumeration of each class.

Paradigmatic examples of \cbers are the identity relation on a given standard Borel space $X$, denoted $\Id(X)$, and the eventual equality relation on the Cantor space $2^\omega$, denoted $E_0$.
%which identifies two infinite binary sequences if they are bitwise eventually the same.
It follows from Silver's dichotomy \cite{silver1980counting} that if $X$ is uncountable, then $\Id(X)$ is Borel reducible to any \cber on $X$. Moreover, the Glimm-Effros dichotomy \cite{harrington1990glimm} states that $E_0$ is a successor of $\Id(2^\omega)$ in the Borel hierarchy, i.e., every Borel equivalence relation is either  reducible to $\Id(2^\omega)$ or $E_0$  reduces to it. Beyond $E_0$, the Borel hierarchy of cbers is much wilder: e.g., Loveau and Velickovic \cite{louveau1994note} proved that it contains both infinite chains and antichains. 
%$\langle \mathcal{P}(\omega), \subseteq^*\rangle$ ordered by almost inclusion embeds into it. 
Yet, there exists a universal \cber $E_\infty$ to which all  \cbers  reduce \cite{dougherty1994structure}.
%one of its realization is the orbit relation generated by the shift action of  the free group with $2$ generators $\mathbb{F}_2$ on $2^{\mathbb{F}_2}$.

Coskey, Hamkins, and Miller \cite{CHM} suggested to effectivize set theoretic Borel equivalence relations by restricting the focus to their computably enumerable (c.e.) instances. By identifying c.e.\ sets with their indices, this restriction allows to project equivalence relations from $2^\omega$ to $\omega$. 
% indeed, if $E$ is a Borel equivalence relation on $2^\omega$ (and we identify sets with their characteristic functions), then $E^{ce}$ is an equivalence relation on $\omega$ given by
%\[
%e \, E^{ce} \, i \Leftrightarrow W_e \, E \, W_i,
%\] 
%where $(W_e)$ denotes, as usual, a uniform enumeration of all c.e.\ sets. 
So, e.g., $\Id(2^\omega)$ and $E_0^{ce}$ translate, respectively, to the equality of  c.e.\ sets, denoted by $=^{ce}$, and to the almost equality of c.e.\ sets, denoted by $E_0^{ce}$. 

%Then, one may study  various $E^{ce}$'s by means of computable reducibility, the most popular tool for classifying countable equivalence relations. 
In \cite{CHM}, two effective   analogues of the class of \cbers are proposed. Roughly (formal definitions will be given below): the \emph{c.e.\ orbit equivalence relations}, which are based on effectivizing the Feldman--Moore theorem, are those  arising from  a computable group acting, in a suitable way, on the c.e.\ sets; the \emph{equivalence relations \eiti}, on the other hand,  are based on effectivizing the Lusin--Novikov theorem.
%, are the quotients of $=^{ce}$ that have a uniformly computable enumeration of each class.

Coskey, Hamkins, and Miller \cite{CHM} proved  that, contrary to the Borel case, the two notions do not align. They showed that the equivalence relation $E^{ce}_0$ is \eiti but no suitable action on the c.e.\ sets realizes it. They asked whether $E^{ce}_0$ is computably bireducible with a c.e.\ orbit equivalence relation. 
More generally, they asked for a degree theoretic analysis of these notions under computable reducibility, the most popular tool for classifying equivalence relations on $\omega$.
% (see, e.g., \cite{?}). 
In this paper, we offer such analysis.
Our main theorem expresses a sharp and quite unexpected dichotomy: 

\begin{mainthm}
	Every c.e.\ orbit equivalence relation is computably bireducible with exactly one of $=^{ce}$ or $E_0^{ce}$.
	%
	%Up to computable reducibility, every c.e.\ orbit equivalence relation is either equivalent to $=^{ce}$ or $E_0^{ce}$.
\end{mainthm}

Hence, c.e.\ orbit equivalence relations are much more well-behaved than their Borel counterpart. This  is in sharp contrast with the evidence that many desirable properties of a poset fail for degree structures based on  computable reducibility (such as \textbf{Ceers} \cite{GaoGerdes,JoinsAndMeets,TheoryOfCeers} and \textbf{ER} \cite{andrews2021structure}). On the other hand, the property of being \eiti gives rise to a more complicated hierarchy: Theorems \ref{infinite chains of eqrels enumerable in indices} and \ref{infinite antichain of eqrels enumerable in indices} state that, between $=^{ce}$ and $E_0^{ce}$, there are both  infinite chains and infinite antichains of equivalence relations which
are enumerable in the indices. It follows that there is no analogue of the Glimm-Effros dichotomy for equivalence relations on the c.e.\ sets, which gives a strong solution to \cite[Question 3.5]{CHM}.

In Section 1.1, we recall some definitions for working with group actions and prove some easy and useful Lemmas about group actions. We also give the definition of when a group action on the c.e.\ sets is computable in indices, and define when an equivalence relation on the c.e.\ sets is a \ceoer or is \eiti.
In Section 2, we show that any action on $\mathbf{CE}$ which is computable in indices is induced by a computable permutation group acting on $\omega$. Using this, in Section 3 we prove the dichotomy theorem that every c.e.\ orbit equivalence relation is equivalent to either $\Ece$ or $E_0^{ce}$. To do this, we show that every action comes in one of three types and we prove the result for each of these types in subsections 3.1-3.3. Finally, in Section 4 we consider the equivalence relations \eiti and show that there are infinite chains and antichains of these between $\Ece$ and $E_0^{ce}$.

\subsection{Preliminaries} We assume that the reader is familiar with the fundamental
notions and techniques of computability theory. In particular, we shall freely use the standard machinery for  priority arguments, 
(e.g., strategies, requirements, outcomes, injury, tree of strategies), 
as is surveyed in \cite{Soare}.
% In particular, for nodes $\alpha$ and $\beta$ on a tree of strategies, the notation $\alpha <_L \beta$ means that $\alpha$ is to the left of $\beta$.
 
 \subsubsection*{Group Actions}
Let $G$ be a group acting on some set $X$. Let $\pi:G\rightarrow S_X$ be the induced permutation representation.
We say \emph{$G$ has only finitely many actions} if $\ran(\pi)$ is finite.
For any $Y\subseteq X$, we let $\Stab(Y)=\{g\in G : \forall x\in Y (g\cdot x=x)\}$. We say 
%\emph{$G$ has isolated actions} 
\emph{the actions of $G$ are isolated}
if there is a finite set $F\subseteq X$ so that $\Stab(F)= \ker(\pi)$. 
Otherwise, we say \emph{the actions of $G$ are non-isolated}.
The \emph{orbit equivalence relation} $E_G$ on $X$ is given by
$
x \, E_G \, y \Leftrightarrow (\exists g \in G)(g \cdot x =y).
$
Equivalence classes of $E_G$ are called \emph{$G$-orbits}. For a set $Y\subseteq X$, we let $G\cdot Y = \{g\cdot x : g\in G,\, x\in Y\}$. Similarly, if $Y\subseteq X$ and $g\in G$, we let $g\cdot Y = \{g\cdot x : x\in Y\}$. If $Y$ is contained in a single $G$-orbit, we say that $G$ acts \emph{transitively} on $Y$.

The next few easy group theoretic lemmas will facilitate our classification of the c.e.\ orbit equivalence relations.

\begin{lem}\label{equivalence of non-isolation notions}
	If the actions of $G$ are non-isolated, then for any $g\in G$ and $F$ a finite subset of $\omega$, there is $h\in G$ so that $h\restriction F = g\restriction F$ and $h\circ g^{-1}\notin \ker(\pi)$.
\end{lem}
\begin{proof}
	Let $h_0\in \Stab(F)\smallsetminus \ker(\pi)$ and let $h=h_0\circ g$.
\end{proof}

\begin{lem}\label{trichotomy}
If the actions of $G$ are isolated and all $G$-orbits are finite, then $G$ has only finitely many actions.
\end{lem}
\begin{proof}
Let $F$ be finite so $\Stab(F) = \ker(\pi)$. Any $f,g\in G$ that act the same on $F$ have $f^{-1}g\in \ker(\pi)$, so $\pi(f)=\pi(g)$. But since each $x\in F$ has a finite $G$-orbit, there are only finitely many total possible images for $F$ for any action in $G$.
\end{proof}

\begin{lem}\label{lemma: avoiding F}
	Suppose $G$ acts on some infinite set $S$ transitively. Then, for any finite set $F\subseteq S$, there is  $g\in G$ so that $g\cdot F$ is disjoint from $F$.\end{lem}
\begin{proof}
	For each pair $x,y\in F$, the set of $g\in G$ so that $g\cdot x = y$ is a coset of $\Stab(x)$. Then, since $S$ is infinite and $G$ acts transitively, the $\Stab(x)$ must have infinite index. By Neumann's lemma \cite{neumann1954groups}, a group cannot be covered by a finite union of cosets of subgroups of infinite index.\footnote[3]{We thank Meng-Che Ho for pointing out this slick proof.}
\end{proof}

%
%
% naturally give rise to equivalence relations: the \emph{orbit equivalence relation} $E_G$, generated by a group $G$ acting on a domain $X$, is given by
%\[
%x \, E_G \, y \Leftrightarrow (\exists \gamma \in G)(\gamma \cdot x =y).
%\]

\subsubsection*{Computable Reducibility}

%More generally, we focus on equivalence relations that can be represented on $\omega$. 
For equivalence relations $E$ and $F$ on $\omega$,  $E$ is  \emph{computably reducible} to $F$, written  $E\leq_c F$,  if there is a computable function $f$ so that
$
x \, E \, y \Leftrightarrow f(x) \, F \, f(y).
$
Henceforth, we refer to computable reductions as just reductions. We write $E\, \equiv_c F$, if $E$ and $F$ reduce to each other. 
%The notation $f: E\leq_c F$ will mean that $f$ is a reduction of $E$ to $F$.

\begin{defn}[\cite{CHM}]\label{def:core defns}  Let $(W_e)_{e\in\omega}$ be a uniform enumeration of all c.e.\ sets, and denote by $\mathbf{CE}$  the collection of c.e.\ subsets of $\omega$.
\begin{itemize}
\item If $E$ is an equivalence relation on $2^\omega$, then $E^{ce}$ is an equivalence relation on $\omega$ given by
\[
e \, E^{ce} \, i \Leftrightarrow W_e \, E \, W_i.
\] 
 Note that every $E^{ce}$ is a \emph{quotient} of $=^{ce}$, i.e., $=^{ce}\, \subseteq \, E^{ce}$. 
\item An action of a computable group $G$ on $\mathbf{CE}$ (note that the action is on the collection of sets, not on indices for these sets) is
\emph{computable in indices} if there is computable $\alpha: G \times \omega \rightarrow \omega$ so that
\[
W_{\alpha(\gamma,e)}=\gamma \cdot W_{e}.
\]
We use the term \emph{c.e.\ orbit equivalence relation} and the notation $E^{ce}_G$ to mean an orbit equivalence relation of a group action on \textbf{CE} which  is computable in indices.
\item $E^{ce}$ is \emph{\eiti} if there is a computable $\alpha : \omega\times \omega \to \omega$ so that
\[
i \, E^{ce} \, j \Leftrightarrow (\exists n)(W_{\alpha(i,n)}=W_j).
\]
\end{itemize} 
\end{defn}

It is easy to see that $=^{ce}$ reduces to $E_0^{ce}$. To see that $E_0^{ce}$ does not reduce to $\Ece$, it suffices to observe that $=^{ce}$ is $\Pi^0_2$ while $E_0^{ce}$ is strictly $\Sigma^0_3$ (see \cite[Theorem 3.4]{CHM}). In fact, the following holds: 

\begin{thm}[Ianovski, Miller, Ng, Nies \cite{IanovskiMillerNgNies}]\label{e0 is sigma3 complete}
$E_0^{ce}$ is a universal $\Sigma^0_3$ equivalence relation under computable reducibility.
\end{thm}

Note that both the c.e.\ orbit equivalence relations and the equivalence relations \eiti are subclasses of $\Sigma^0_3$ equivalence relations. Thus, they all  reduce to $E^{ce}_0$. Finally, c.e.\ equivalence relations, widely investigated in literature (see, e.g., \cite{GaoGerdes,JoinsAndMeets,TheoryOfCeers,andrews2022investigating}), are called \emph{ceers}.

\section{Reducing to computable permutation groups}
We begin by proving a simple yet fundamental lemma that describes how the Recursion Theorem constrains the behavior of group actions which are computable in indices. From this lemma it will follow that, without loss of generality, we may assume that any \ceoer is naturally induced by a computable permutation group on $\omega$ 
%(i.e., a computable subgroup of $S_\infty$).
(i.e., a computable group acting computably by permutations on $\omega$)

\begin{nota}
Throughout this section, we let $\alpha$ be a computable function witnessing that a given computable group $G$ acts computably in indices on $\mathbf{CE}$ (see the second bullet of Definition \ref{def:core defns}).
\end{nota}

\begin{lem}\label{lem: constraints to orbit eqrels computable in the indices}
	For  each $\gamma\in G$ and c.e.\ sets $U,V$
	\begin{enumerate}
		\item $U \subseteq V \Rightarrow \gamma \cdot U \subseteq \gamma \cdot V$;
		\item %$\alpha$ is cardinality invariant, i.e.,
		 $|V| =|\gamma \cdot V|$.
	\end{enumerate}
\end{lem}

\begin{proof}
	$(1)$ Suppose that $U \subseteq V$ and take any $n\in \gamma \cdot U$. Let $e$ be an index we control by the Recursion Theorem\footnote{Formally speaking, we describe a construction of a c.e.\ set $W_{f(e)}$ from a given index $e$ and the recursion theorem gives us an index so $W_{f(e)}=W_e$.}. We copy $U$ into $W_e$, unless we see $n$ enter in $W_{\alpha(\gamma,e)}$ in which case we copy $V$ in $W_e$. We must have $n\in W_{\alpha(\gamma,e)}$ as otherwise $W_e=U$ and $n\in \gamma\cdot U = W_{\alpha(\gamma,e)}$. Thus $n\in W_{\alpha(\gamma,e)}$ and $W_e=V$. This shows $n\in W_{\alpha(\gamma,e)} = \gamma \cdot V$. 
	
	$(2)$ Suppose that $|V|> |\gamma \cdot V|$. Since the action of $\gamma$ must be injective on $\mathbf{CE}$ and $(1)$ holds, there is simply not enough room to accommodate all subsets of $V$ into the subsets of $\gamma \cdot V$. To exclude that $|V|< |\gamma \cdot V|$, just note that $V = \gamma^{-1}\cdot (\gamma \cdot V)$. 
	%The reasoning for complements is symmetric (I guess!).
\end{proof}

\begin{defn}
For each $\gamma\in G$, let the function $F_\gamma:\omega \to \omega$ be given by 
\[ 
F_\gamma(n)= m \Leftrightarrow \gamma\cdot \{n\}=\{m\}.
\] 
\end{defn}

%\begin{defn}
%	Let $\alpha$ be a computable function witnessing that some $G$ acts computably in the indices on $\mathbf{CE}$. Then, for each $\gamma\in G$, we define a function on $\omega$ as follows $F(g,n)= m$ if whenever $W_i = \{n\}$, we have $W_{\alpha(g,i)} = \{m\}$. 
%\end{defn}

\smallskip

\begin{lem}\label{ce orbit eqrels are permutation induced}
	For all $\gamma\in G$, $F_\gamma$ is a computable permutation of $\omega$. Moreover, $\gamma \cdot V = \{F_\gamma(n) : n\in V\}$ for each c.e.\ $V$.
\end{lem}
\begin{proof}
	
	We first observe that $F_\gamma$ is a permutation of $\omega$. Since $\gamma$ acts injectively on $\mathbf{CE}$, $F_\gamma$ must be injective. But since $F_\gamma$ is necessarily the inverse of $F_{\gamma^{-1}}$, we see it is also surjective. This permutation of $\omega$ is computable since we can just wait to see which number enters $W_{\alpha(\gamma,i)}$ where $i$ is any index such that $W_i = \{n\}$.
	
	Next, by Lemma \ref{lem: constraints to orbit eqrels computable in the indices}(1), $\gamma \cdot V \supseteq \{F_\gamma(n) : n\in V\}$.
	Applying the same to $\gamma^{-1}$, we see that $\gamma \cdot V = \{F_\gamma(n) : n\in V\}$.
\end{proof}

This allows us consider the c.e.\ orbit equivalence relations in a more concrete fashion:

\begin{defn}
	For $G$ a computable subgroup of $S_\infty$, let 
	\[
	i \mathrel{R^{ce}_G} j \Leftrightarrow (\exists \gamma \in G)(W_i = \{ \gamma(x) : x\in W_j\}).
	\]
\end{defn}

The next lemma, which follows directly from Lemma \ref{ce orbit eqrels are permutation induced}, ensures that focusing only to c.e.\ orbit equivalence relations of the form $R^{ce}_G$ is not restrictive:
\begin{lem}\label{normal form c.e. orbit eqrels}
	For every \ceoer $E^{ce}_G$, there is a computable subgroup $H$ of $S_\infty$  so that $E^{ce}_G=R^{ce}_H$.
\end{lem}

%Instead of abstractly considering a \ceoer of a group acting on $\mathbf{CE}$, from now on, we will fix a computable subgroup $H$ of $S_\infty$ and let
%\[
% i \mathrel{R^{ce}_H} j \Leftrightarrow (\exists \gamma \in H)(W_i = \{ \gamma(x) : x\in W_j\}).
% \]
%Lemma \ref{ce orbit eqrels are permutation induced} ensures that every $E^{ce}_G$ is equal to $R^{ce}_H$ for some computable permutation group $H$.

%\begin{defn}
%	For $G$ a computable permutation group on $\omega$, we let $E_G$ be the equivalence relation defined by $i \mathrel{E_G} j$ if and only if there is some $g\in G$ so that $W_i = g[W_j]$. That is $W_i = \{ g(x)\mid x\in W_j\}$. 
%\end{defn}

\section{The dichotomy theorem for c.e.\ orbit equivalence relations}

This section is devoted to the proof of the dichotomy theorem: We show that every $R^{ce}_G$ is either $\Sigma^0_3$-universal (and thus, by Theorem \ref{e0 is sigma3 complete}, equivalent to $E_0^{ce}$) or is equivalent to $\Ece$. First, we note that every $R^{ce}_G$ lies above $\Ece$.

\begin{thm}\label{Always above Ece}\label{G finite}
	There is a reduction $f$ of  $\,\Ece$ to itself so that, if $W_i\neq W_j$, then $W_{f(i)}$ is not computably isomorphic to $W_{f(j)}$. In particular, 
%for any computable group $G$ of computable permutations, 
	we have that $f$  reduces $\Ece$ to any $ R^{ce}_G$.
\end{thm}
\begin{proof}
	We will construct a sequence of sets $(V_k)_{k\in\omega}$ so that $W_i = W_j$ implies $V_i=V_j$ and $W_i\neq W_j$ implies $V_i$ and $V_j$ are not computably isomorphic. To do so, we shall satisfy the following requirements:
\begin{itemize}
\item[$\mathcal{R}_{i,j}$]: If $W_i = W_j$, make $V_i = V_j$;
\item[$\mathcal{D}^n_{i,j}$]: Make $V_j\neq \phi_n(V_i)$.
\end{itemize}
	
The $\mathcal{R}_{i,j}$-strategies have two outcomes: $\infty<f$. Similarly, $\mathcal{D}^n_{i,j}$-strategies have  outcomes: $d<w$. We place these outcomes on a tree of strategies $\mathbf{T}$ meeting the following conditions:
 Every path contains an $\mathcal{R}_{i,j}$-node $\alpha$ before any $\mathcal{D}^n_{i,j}$ strategy. Every path extending $\alpha\smallfrown f$ contains strategies for $\mathcal{D}^n_{i,j}$, for each $n$. No $\mathcal{D}^n_{i,j}$-strategy extends $\alpha\smallfrown \infty$.
 
\subsection*{The strategy to meet $\mathcal{R}$-requirements}	For $\mathcal{R}_{i,j}$-strategies, we use the usual computable approximation to determine if $W_i=W_j$. When the length of agreement of $W_{i,s}$ and $W_{j,s}$ changes, we take outcome $\infty$ and act as follows: we replace $V_{i,s}$ by $V_{i,s}\cup (V_{j,s}\cap [0,\ell])$ and $V_{j,s}$ by $V_{j,s}\cup (V_{i,s}\cap [0,\ell])$, where $\ell$ denotes the length of agreement of $V_{i,s}$ and $V_{j,s}$.
 %\footnote{the maximal value $\leq s$ so that these two sets agree up to this length}, 
On all other stages, let $x$ be the least element such that $W_{i,s}(x)\neq W_{j,s}(x)$. We say this is an \emph{$(i,j)$-stage} if $x\in W_{i,s}\smallsetminus W_{j,s}$ and it is a \emph{$(j,i)$-stage} if $x\in W_{j,s}\smallsetminus W_{i,s}$.

\subsection*{The strategy to meet $\mathcal{D}$-requirements}	A $\mathcal{D}^n_{i,j}$-strategy $\alpha$ acts as follows: 
Assume that this is an $(i,j)$-stage (if it is a $(j,i)$-stage instead, reverse the role of $i$ and $j$ using $\phi_n^{-1}$ instead of $\phi_n$).
 Also, suppose there are $M$ numbers restrained by higher priority $\mathcal{D}$-strategies. 
 
	First, we choose $M+1$ new numbers $K_m^\alpha$ for $m<M+1$ and restrain $K_m^\alpha$ from entering any set $V_k$. We wait for $\phi_n(K_m^\alpha)$ to converge  for each $m$ and take outcome $w$. Once $\phi_n(K_m^\alpha)$ converges for each $m$, we take outcome $d$ and act as follows:
	If there are $m_0<m_1<M+1$ so that $\phi_n(K_{m_0}^\alpha)=\phi_n(K_{m_1}^\alpha)$, then we do nothing since $\phi_n$ is not a computable permutation. Otherwise, we choose one number $\phi_n(K^\alpha)$ which is not restrained by a higher priority $\mathcal{D}$-strategy. We act depending on the value of $\phi_n(K^\alpha)$:
	
	\begin{itemize}
		\item[$(a)$] If $\phi_n(K^\alpha) = K_e^\beta$ for a lower priority $\beta$ (i.e., a strategy $\beta$ being injured by our taking outcome $d$) or if $\phi_n(K^\alpha)$ is not chosen as $K_e^\beta$ for any $\beta$, then we place $\phi_n(K^\alpha)$ into $V_j$.
		\item[$(b)$] If $\phi_n(K^\alpha) = K^\alpha$, then we place $K^\alpha$ into $V_i$ (this is the case where it being an $(i,j)$-stage matters).
	\end{itemize}
	
	\subsection*{The verification.} The verification is based on the following lemmas.
	
	\begin{lem}\label{restraintsaremaintainedhere}
		If $\alpha$ places a restraint against $K^\alpha_m$ entering any set, and $\alpha$ has not acted or been injured, then $K^\alpha_m$ has not entered any set.
				If $\alpha$ acts under case $(a)$ and is never injured, then $K^\alpha_m$ still never enters any set.
	\end{lem}
	\begin{proof}
		Since $K^\alpha_m$ is not in any $V_i$, it cannot enter any $V_i$ via action for an $\mathcal{R}$-strategy. Only a higher priority $\mathcal{D}$-strategy or $\alpha$ itself would put $K^\alpha_m$ into any set. In the former case, $\alpha$ would be injured by this, and in the latter case, $\alpha$ acts. If $\alpha$ acts under case $(a)$, then $\alpha$ also does not put $K^\alpha_m$ into any $V_i$.
	\end{proof}
	
	\begin{lem}\label{AwkwardChainingWin}
		If $\alpha$ puts $K^\alpha_m$ into $V_i$ by case $(b)$, then either $\alpha$ is injured or $K^\alpha_m$ never enters $V_j$ (symmetrical if the stage is a $(j,i)$-stage).
	\end{lem}
	\begin{proof}
		For $K^\alpha_m$ to enter $V_j$, there must be some sequence of sets $V_{k_0},\ldots V_{k_n}$ so that $V_i = V_{k_0}$, $V_j = V_{k_n}$ and at some stage after stage $s$, we must approximate that $W_{k_{m}} = W_{k_{m+1}}$ with lengths of agreement at least $K^\alpha_m$. But consider the length of agreement $\ell$ of $W_{i,s}$ and $W_{j,s}$. We had $\ell\in W_{i,s}\smallsetminus W_{j,s}$ since it was an $(i,j)$-stage. If the length of agreement changed, then $\alpha$ would have been injured. But then we must see $\ell$ enter each successive $W_{k_{m+1}}$ until we see $\ell$ enter $W_{k_n}=W_{j}$, so the length of agreement would change after all.
	\end{proof}
	
	\begin{lem}
		If $\alpha$ is on the true path, then $\alpha$ ensures that its requirement is satisfied.
	\end{lem}
	\begin{proof}
		If $\alpha$ is an $\mathcal{R}$-requirement, it need only succeed if the true outcome is infinite. In this case, as the length of agreement between $W_{i,s} $ and $W_{j,s}$ goes to infinity, we ensure that more and more of $V_i$ and $V_j$ agree on a cofinal set of stages, ensuring that $V_i=V_j$.
		
		Next, suppose that $\alpha$ is a $\mathcal{D}^n_{i,j}$-requirement. We consider a stage $s$ late enough that $\alpha$ is never injured after $s$ and $\alpha$ acts at stage $s$ if it ever will. If the true outcome is $w$, then $\phi_n$ is not a permutation and the requirement is satisfied. If the true outcome is $d$, we must consider the two cases above: In case $(a)$, 
		Lemma \ref{restraintsaremaintainedhere} shows that $\phi(K^\alpha_m)\in V_j\smallsetminus V_i$ and in case $(b)$ Lemma \ref{AwkwardChainingWin} shows that $\phi(K_m^\alpha)\in V_i\smallsetminus V_j$ (symmetrical if it is a $(j,i)$-stage).
	\end{proof}
%	This concludes the proof that $V_i$ (i.e., $W_{f(i)}$) and $V_j$ (i.e., $W_{f(j)}$) are computably isomorphic if and only if $W_i$ and $W_j$ coincide. 
	This concludes the proof that $f$ reduces $=^{ce}$ to any $R^{ce}_G$.
\end{proof}

We just proved that  $=^{ce}$ is the least c.e.\ orbit equivalence relation. To calculate the complexity of all $R^{ce}_G$'s  and obtain the dichotomy result, we shall now separate three cases: $(i)$ $G$ has only finitely many actions; $(ii)$ There is an infinite $G$-orbit; $(iii)$ The actions of $G$ are non-isolated. Lemma \ref{trichotomy} guarantees that there are no other cases to be considered.

%
%Thus we will consider the following 3 cases in the following three sections:
%\begin{itemize}
%	\item $G$ contains only finitely many actions.
%	\item There is an element with an infinite $G$-orbit.
%	\item The actions in $G$ are not isolated.
%\end{itemize}

\subsection{Case $(i)$: $G$ has only finitely many actions} To facilitate our analysis, we introduce the following equivalence relation which turns out to be equivalent to $=^{ce}$. 

\begin{defn}
Let $\Esetn$ be given by $i \mathrel{\Esetn} j$ if and only if $i=\langle i_0,\ldots i_{n-1}\rangle$ and $j= \langle j_0\ldots j_{n-1}\rangle$
	and $\{W_{i_k}: k< n\} = \{W_{j_k}: k< n\}$ where $\langle \cdot , \ldots ,\cdot\rangle$ is an $n$-ary pairing function.
\end{defn}
	%\footnote{Here we are using a standard pairing function giving a bijection between $\omega$ and $\omega^n$}

\begin{lem}\label{FinGBelowEsetn}
	If $G$ has only finitely many actions, then $R^{ce}_G$ reduces to $\Esetn$ for some $n$.
\end{lem}
\begin{proof}
	Let $g_0,\ldots g_{n-1}$ be group elements representing all distinct actions in $G$.
	We obtain a reduction of $R^{ce}_G$ to $\Esetn$ by the map which sends any c.e.\ set $W_i$ to an index for the family $\{g_0\cdot W_i,\ldots , g_{n-1}\cdot  W_i\}$. 
\end{proof}

We now need to show that $\Esetn$ reduces to  $\Ece$. 

\begin{thm}\label{thm:EsetnReducesToEce}
	For each $n$, $\Esetn$ reduces to $\Ece$.
\end{thm}
\begin{proof}
	Let $h$ be a function which sends $i=\langle i_0,\ldots , i_{n-1}\rangle$ to a c.e.\ index for the set $$V_i=\{(k,\rho_0,\ldots , \rho_{n-1}) : \text{ each } \rho_i \in 2^k\text{ and } (\exists \pi\in S_n)(\forall i<k)(\rho_i\subseteq W_{\sigma(i)})\}.$$
	
	It is immediate that if $i \mathrel{\Esetn} j$ then $V_i = V_j$. On the other hand, if $V_i = V_j$, then for every $k$ there is some $\pi_k\in S_n$ so that $(\forall l<n)(W_{i_l}\restriction k = W_{j_{\pi_k(l)}}\restriction k)$. Hence, by the pigeonhole principle, there is some permutation $\pi \in S_n$ so that $(\forall l<n)(W_{i_l} = W_{j_{\pi(l)}})$.
\end{proof}

Putting together Lemma \ref{FinGBelowEsetn}, Theorem \ref{thm:EsetnReducesToEce}, and Theorem \ref{Always above Ece}, we get that c.e.\ orbit equivalence relations induced by finite permutation groups are as simple as possible:

\begin{thm}
	If $G$ has only finitely many actions, then $R^{ce}_G\equiv_c \Ece$.
\end{thm}

\subsection{Case $(ii)$: There is an infinite $G$-orbit}

\begin{thm}\label{Inf orbit}
	If there is an infinite $G$-orbit, then $R^{ce}_G$ is $\Sigma^0_3$-universal.
	
	%If there is an element $n\in \omega$ so that the $G$-orbit of $n$ is infinite, then $R^{ce}_G$ is $\Sigma^0_3$-universal.
\end{thm}
\begin{proof}
	We will do our coding entirely within a single infinite orbit, which forms an infinite c.e.\ set. Using a computable bijection between this set and $\omega$, we assume with no loss of generality that $G$ acts transitively on $\omega$.
	
	We fix $R$ a $\Sigma^0_3$ equivalence relation defined by $i\mathrel R j$ if and only if $\exists n X(i,j,n)$ for a $\Pi^0_2$ relation $X$. We construct a uniformly c.e.\ sequence of sets $V_i = W_{f(i)}$ for $i\in \omega$ so that $i \mathrel R j$ if and only if $f(i) \mathrel{R^{ce}_G} f(j)$. That is, there is an $n$ so that $X(i,j,n)$ if and only if there is some $g\in G$ so that $V_i = g(V_j)$. 
	
\subsection*{The requirements and their interaction}	For every $i<j\in \omega$ and $n\in \omega$, we have the following requirements:
	
	\begin{itemize}
	\item[$\mathcal{P}_{i,j}^n$]: If $X(i,j,n)$, then make $V_j = g(V_i)$ for some $g\in G$. If $\neg X(i,j,n)$ then ensure that either $\phi_n$ is not a permutation of $\omega$ or $V_j\neq \phi_n(V_i)$.
	\end{itemize}
	
	Each strategy has three outcomes: $\infty<d<w$. Outcome $\infty$ represents $X(i,j,n)$, outcome $d$ represents $\neg X(i,j,n)$ and we succeed in diagonalizing to ensure $V_j\neq \phi_n(V_i)$, and outcome $w$ represents $\neg X(i,j,n)$ and we never get a chance to diagonalize since we are waiting for $\phi_n$ to converge. 
	
	We put these strategies on a tree so that, if $\tau$ is given the requirement $\mathcal{P}_{i,j}^n$, then we place no strategy $\mathcal{P}_{i',j}^k$ or $\mathcal{P}_{j,i'}^k$ below $\tau\smallfrown \infty$. We do this so that for every $\sigma\in \{\infty,d,w\}^\omega$ and $j\in \omega$, either there is precisely one $\tau$ so that $\tau\smallfrown \infty \prec \sigma$ and $\tau$ is a $\mathcal{P}_{i,j}^k$-strategy for some $i<j$ and $k$, or for every $i<j$ and $k$, there is some $\tau$ so that $\tau$ is a $\mathcal{P}_{i,j}^k$-strategy and either $\tau\smallfrown d\prec \sigma$ or $\tau\smallfrown w\prec \sigma$.

	When first visited, a $\mathcal{P}^n_{i,j}$-strategy will choose an element $g\in G$ which it will use in its infinite outcome. Its choice of $g$ must be consistent with the rest of the construction. In particular, if it applies the infinite outcome, it does not want to cause injury to any higher priority requirement's diagonalization. Namely, for the purpose of diagonalizing, a strategy $\mathcal{P}^{n'}_{i',j'}$ will choose a number $K$ and it may put $K$ into $V_{i'}$ and attempt to keep $\phi_{n'}(K)$ out of $V_{j'}$. On the other hand, under the infinite outcome, for the purpose of ensuring $V_j=g(V_i)$, $\mathcal{P}_{i,j}^n$ will act by putting $g^{-1}(V_j)$ into $V_i$ and putting $g(V_i)$ into $V_j$. We must ensure that the cumulative effect of these infinite outcomes will not ruin the $\mathcal{P}_{i',j'}^{n'}$-strategy diagonalization. That is, that they will not put $\phi_{n'}(K)$ into $V_{j'}$.

	To ensure this, we will define the set $F$ of numbers currently relevant to the construction. That includes all those $K$'s which might enter some set for the sake of a diagonalization and also all those numbers $N$ so that $N$ entering any $V_\ell$ might possibly cause $\phi_n(K)$ to enter $V_{j'}$ by the actions of all the currently active strategies. Then, we will rely on Lemma \ref{lemma: avoiding F} to choose an element $g\in G$ so that $g\cdot F \cap F = \emptyset$.

	\begin{defn}
		At all stages $s$ of the construction, for any given node $\alpha$, we define the set of $\alpha$-restrained pairs as follows:
		If $\alpha$ restrains a number $n$ from entering a set $V_j$, then the pair $(n,j)$ is a restrained pair. In addition, we say a pair $(m,k)$ is \emph{$\alpha$-restrained} if there is a sequence of currently active nodes on the tree $\beta_0,\ldots \beta_n$ such that the following hold:
		\begin{itemize}
			\item For each $i$, either $\beta_i\smallfrown \infty \preceq \alpha$ or $\alpha$ is to the left of $\beta_i$ or $\alpha\smallfrown d\preceq \beta_i$ or $\alpha\smallfrown w \preceq \beta_i$.
			\item If $i<j$, then $\beta_i$ taking outcome $\infty$ does not injure $\beta_j$.
			\item If $m$ were in $V_k$ and then each of these $\beta_i$ were to take their infinite outcomes, in order, it would cause $n$ to enter $V_j$.
		\end{itemize}
	\end{defn}

	%		\begin{defn}
	%			At stage $s$, the collection of pairs of concern are defined as follows:
	%			If some requirement has placed a restraint of the form ``$n$ should not enter $V_i$'' for some $i$, then $(n,i)$ is a pair of concern. 
	%			If $m$ is a number so that putting $m$ into $V_j$ could cause $n$ to enter $V_i$ by actions of infinite outcomes of currently active requirements, then $(m,j)$ is also a number of concern.
	%		\end{defn}
	%	\texttt{Formalize -- e.g., these requirements act in turns and can't have been injured first, etc.}

	\subsection*{The strategy to meet $\mathcal{P}^n_{i,j}$ at node $\tau$}  When initialized, let $F$ be the set of numbers mentioned so far in the construction, including every $n$ which is in an $\alpha$-restrained pair for any $\alpha$ and every requirement's parameter $K$ or $\phi_n(K)$ (if already defined). We will show in Lemma \ref{lem: finite restraints} that this is a finite set. Applying Lemma \ref{lemma: avoiding F}, choose a $g$ in $G$ so that $g\cdot F\cap F = \emptyset$.
	
\begin{itemize}
\item When we approximate that $X(i,j,n)$ holds, this strategy will attempt to make $V_j=g(V_i)$. In this case, we take the outcome $\infty$ and the node acts by putting $g^{-1}(V_j^s)$ into $V_i$ and $g(V_i^s)$ into $V_j$; 
\item When we approximate that $\neg X(i,j,n)$ holds, we employ the following strategy towards
	ensuring $V_j\neq \phi_n(V_i)$. We first choose a new number $K$, in particular $K$ is not in any $\alpha$-restrained pair for any currently active $\alpha$. Moreover, after placing this restraint, there should be no $\tau$-restrained pair $(m,l)$ with $m\in V_l$. We prove in Lemma \ref{K exists} that such a $K$ can be chosen. We place the restraint that $K$ should not enter $V_i$, and wait for $\phi_n(K)$ to converge. While we wait, we take outcome $w$.
	
	Once we see $\phi_n(K)$ converge, we check if we can place a restraint keeping $\phi_n(K)$ from entering $V_j$. That is, we check if, once we place this restraint we would have an $\alpha$-restrained pair $(m,l)$ with $m$ already in $V_l$. If so, we do nothing. If not, we place $K$ into $V_i$ and we place the restraint that $\phi_n(K)$ should not enter $V_j$. In this case, we take outcome $d$.
\end{itemize}

	\subsection*{Verification.} The verification is based on the following lemmas.
	
	\begin{lem}\label{lem:InfGOrbitInfOutcome}
		If $\alpha$ is a $\mathcal{P}^n_{i,j}$-strategy on the true path with true outcome $\infty$, then $i \mathrel{E_G} j$.
	\end{lem}
	\begin{proof}
		Let $s$ be a stage when $\alpha$ is last initialized. Then $\alpha$ chooses a group element $g$. Infinitely often, when $\alpha$ is visited, it puts $g^{-1}(V_j^s)$ into $V_i$ and $g(V_i^s)$ into $V_j$. This ensures that $V_j=g(V_i)$ and thus $i\mathrel{E_G} j$.
	\end{proof}

	\begin{lem}\label{lem: finite restraints}
		At each stage $s$, there are only finitely many $\alpha$-restrained pairs.
	\end{lem}
	\begin{proof}
		We note that if $\beta\succeq \gamma\smallfrown \infty$, then we do not have $\gamma$ as a $\mathcal{P}^k_{i,j}$-strategy and $\beta$ as a $\mathcal{P}^l_{i',j}$-strategy or $\mathcal{P}^l_{j,m}$-strategy. Thus, if we put an element into $V_j$ for some $j$, and via some sequence of infinite outcomes, that causes some other element to appear in $V_j$, this must have been due to some injury among the requirements. As there are only finitely many currently active requirements, this can happen only finitely often. Thus, among the finitely many sets $V_i$ currently under consideration, the set of pairs $(m,j)$ which might cause $\alpha$'s restrained number $n$ to enter the set $V_j$ that it is restrained from, is a finite set.
	\end{proof}
	
	\begin{lem}\label{K exists}
		At any stage $s$, active node $\alpha$, and $i\in \omega$, there are only finitely many $K$ so that $\alpha$ placing a restraint against $K$ entering $V_i$ would cause there to be an $\alpha$-restrained pair $(x,m)$ with $x$ already in $V_m$. 
	\end{lem}
	\begin{proof}
		As in Lemma \ref{lem: finite restraints}, each $x$ entering $V_m$ can cause at most finitely many numbers to enter $V_i$ via a sequence of infinite outcomes of currently active nodes. As there are only finitely many numbers at stage $s$ already in $\bigcup_{m\in \omega}V_m$, this makes only finitely many $K$ have the property that a restraint against $K$ entering $V_i$ would cause there to be an $\alpha$-restrained pair $(x,m)$ with $x$ already in $V_m$. 
	\end{proof}
	
	\begin{lem}\label{InductionPreserveRestraints}
		At every stage of the construction, if $\alpha$ is an active node restraining $a$ from entering to $V_b$, then there is no $\alpha$-restrained pair $(m,l)$ so that $m$ is in $V_l$.
		
		Also, if $(m,l)$ is an $\alpha$-restrained pair so that $m=K^\beta$ where 
		$\beta$ is a $\mathcal{P}^c_{l,l'}$-strategy, then $\beta\smallfrown d$ is to the left of $\alpha\smallfrown d$.
		
		%Also, there is no $\alpha$-restrained pair $(m,l)$ so that $m=K^\beta$ where 
		%$\beta$ is a $\mathcal{P}^c_{l,l'}$-strategy so that $\beta\smallfrown d$ is not to the left of $\alpha\smallfrown d$.
	\end{lem}
	\begin{proof}
		
		We prove both claims by simultaneous induction.
		We begin our induction at the moment when $\alpha$ places its restraint. At this point, no $\alpha$-restrained pair $(m,k)$ can have $m$ in $V_k$. If the $\alpha$-restrained pair is $(K,i)$, this is true because $K$ is chosen to be new. If the $\alpha$-restrained pair is $(\phi_n(K),j)$, we ensure this condition before placing the restraint. The second condition is ensured at this moment because when $\alpha$ places its restraint, it takes outcome $w$ or $d$ and thus the only active nodes $\beta$ which currently have a parameter $K^\beta$ must have $\beta\smallfrown d$ to the left of $\alpha\smallfrown d$. 
		
		As moments of the construction go by, we have to check that we preserve the inductive hypotheses. We have four types of actions to consider:
		\begin{enumerate}
			\item New strategies choosing $g\in G$.
			\item Other strategies taking the infinite outcome.
			\item Other strategies taking their outcome $d$.
			\item Other strategies choose a new $K$.
		\end{enumerate}
		
		When a new strategy chooses its $g\in G$, it does so in a way to ensure that it maintains this inductive hypotheses. In particular if $\beta$ is choosing its parameter $g$, then this is the first time $\beta$ is visited since (re)initialization. Thus $\beta$ is the rightmost active node. Thus to violate the inductive hypothesis, $g$ would have to move either an element in one of the $V_j$'s or some $K^\gamma$ to some element which is $\alpha$-restrained. But $g$ is chosen not to do this.
		
		When other strategies take the infinite outcome, either this outcome injures $\alpha$ or this one step was already considered as a possible step before this happened. In particular, if this puts $m$ into $V_l$, then it is because $m'$ was already in $V_{l'}$. Then $(m,l)$ cannot be an $\alpha$-restrained pair, because then $(m',l')$ would have been an $\alpha$-restrained pair contradicting the inductive hypothesis. Thus the first condition is preserved. As there can now be only fewer $\alpha$-restrained pairs and no new parameters $K^\beta$ have been chosen, the second condition is maintained as well.
		
		When other strategies take their outcome $d$, they may put their number $K$ into their set $V_i$. By the second condition of the inductive hypothesis, we have preserved the first condition of the induction hypotheses. As there can now be only fewer $\alpha$-restrained pairs and no new parameters $K^\beta$ have been chosen, the second condition is maintained as well.	
		
		When other strategies choose a new $K$, they do so in order to maintain these inductive hypotheses. In particular, $K$ is chosen to not be in any $\alpha$-restrained pair. This preserves the second condition, and the first condition is unchanged.
	\end{proof}

	\begin{lem}\label{lem:InfGOrbitFiniteOutcome}
		If $\alpha$ is a $\mathcal{P}^n_{i,j}$-strategy on the true path with true outcome $d$ or $w$, then $V_j\neq \phi_n(V_i)$.
	\end{lem}
	\begin{proof}
		
		If the true outcome is $w$, then $\phi_n$ is not total, so we consider the case where the true outcome is $d$. In this case, there are two possibilities to consider. 
		
		In the first possibility, we keep the restraint $K$ shall not enter $V_i$. This is because we see, when visiting $\alpha\smallfrown d$, that if we were to restrain $\phi_n(K)$ from entering $V_j$ we would already have some $\alpha$-restrained pair $(m,l)$ with $m$ in $V_l$. Since all strategies right of $\alpha$ are reinitialized, as are all strategies below $\alpha\smallfrown w$ or $\alpha\smallfrown d$, this means that the strategies whose infinite outcomes are needed to move $m\in V_l$ to $\phi_n(K)\in V_j$ are just those $\beta$ so that $\beta\smallfrown \infty \preceq \alpha$. Since we assume $\alpha$ is on the true path, the outcome $\beta\smallfrown \infty$ will occur infinitely often,  eventually we will see $\phi_n(K)\in V_j$. By the first part of Lemma \ref{InductionPreserveRestraints}, we will never see $K$ enter $V_i$, so we have successfully diagonalized.
		
		In the second possibility, we put $K$ into $V_i$, and we place restraint $\phi_n(K)$ should not enter $V_j$. By the first part of Lemma \ref{InductionPreserveRestraints}, $\phi_n(K)$ never enters $V_j$ and we have successfully diagonalized.
	\end{proof}

	\begin{lem}\label{lem:ItsAReduction}
		For every $i<j$, $i\mathrel R j$ if and only if $f(i) \mathrel{R^{ce}_G} f(j)$. 
	\end{lem}
	\begin{proof}
		We claim by induction that $i\mathrel R j$ if and only if $f(i) \mathrel{R^{ce}_G} f(j)$. We assume the condition for all pairs $(i',j')$ with $i',j'<j$ and consider pairs $(i,j)$ with $i<j$.
	Suppose there is some $i<j$ with $i\mathrel R j$. Then there must be some $i'<j$ so that a $\mathcal{P}_{i',j}^n$-strategy $\tau$ is on the true path with true outcome $\infty$. Then $i'<j$ and $i\mathrel R i'$. By Lemma \ref{lem:InfGOrbitInfOutcome}, the $\tau$-strategy ensures $f(i') \mathrel{R^{ce}_G} f(j)$. By inductive hypothesis, we have $f(i)\mathrel{R^{ce}_G} f(i') \mathrel{R^{ce}_G} f(j)$. Suppose that there is no $i<j$ with $i\mathrel R j$. Then every $\mathcal{P}_{i,j}^n$-strategy with $i<j$ has true outcome $w$ or $d$. Thus Lemma \ref{lem:InfGOrbitFiniteOutcome} ensures that these strategies along the true path ensure that $\phi_n$ is not a bijection between $W_{f(i)}$ and $W_{f(j)}$. Together, these ensure that $f(i)\,\cancel{\mathrel{R^{ce}_G}}\, f(j)$ for each $i<j$.
	\end{proof}
\noindent This concludes the proof that, if   there is an infinite $G$-orbit, then $R^{ce}_G$ is $\Sigma^0_3$-universal.
\end{proof}

\subsection{Case $(iii)$: The actions of $G$ are non-isolated}

We first give a small reduction of our group $G$ which maintains the properties we assume in this case and makes it computable to find indices for $G$-orbits.

\begin{lem}\label{computableorbits}
	If $G$ is a computable group of permutations of $\omega$ so that the actions of $G$ are non-isolated and all $G$-orbits are finite, then there is a computable subgroup $G'$ of $G$ so that the actions of $G'$ are non-isolated and there is a computable function $f$ sending $a\in \omega$ to a canonical index for the $G$-orbit of $a$.
\end{lem}
\begin{proof}
	We construct $G'$ as a c.e.\ subset of $G$ as follows. 
	We act in stages to satisfy requirements: 
	\begin{itemize}
		\item The actions of $G'$ are non-isolated.
		\item The $G'$-orbit of $s$ will not grow after stage $s$.
	\end{itemize}
	
	To satisfy requirements of the second kind, we enumerate the orbit of the first $s$ numbers $a$ at stage $s$. That is, we take the set $G_s$ of elements of $G'$ that we have enumerated by stage $s$. Then, we compute the orbit of $a$ in the finitely generated group generated by $G_s$. Call this set $\mathcal{O}_a$. We then place a restriction that we will, in the future, only consider elements $g$ in $G$ which have the property that $\forall x\in \mathcal{O}_a\,g(x)\in \mathcal{O}_a$. We note that a finite intersection of such subgroups is a subgroup $H_s$ which contains $\text{Stab}(\{0,\ldots n\})$ for some $n$, so $H_s$ has the same properties that the actions of $H_s$ are non-isolated and all orbits are finite.
	
	A requirement of the first kind states that $G'\cap \text{Stab}(\{0,\ldots n\})$ has at least two elements. We find some element other than $1_{G}$ of the restricted group $H_s$ (restricted due to requirements of the second kind) which is in $\text{Stab}(\{0,\ldots, n\})$. We enumerate this element into $G'$. By proceeding as such, we enumerate a subset of $G$ and we let $G'$ be the subgroup generated by these. Thus $G'$ is a c.e.\ subset of $G$, thus itself a computably-presentable group action. Further, since the $G'$-orbit of $n$ does not grow after stage $n$, it follows that it is computable to produce canonical indices for the $G'$-orbits of numbers.
\end{proof}

Now we are ready to handle the case where the actions of $G$ are non-isolated:

\begin{thm}\label{dense actions}
	If the actions of $G$ are non-isolated then $R^{ce}_G$ is $\Sigma^0_3$-universal.
\end{thm}
\begin{proof}
	We may assume that each $G$-orbit is finite, as otherwise the result follows from Theorem \ref{Inf orbit}. We fix $R$ a $\Sigma^0_3$ equivalence relation given by $\exists n X(i,j,n)$ where $X$ is a $\Pi^0_2$ relation. Applying Lemma \ref{computableorbits}, we have a computable subgroup $G'$ of $G$ with uniformly computable orbits. In the proof below, we will construct a sequence of sets $V_i=W_{f(i)}$ and ensure that if $i R j$ then $f(i)\mathrel{E^{ce}_{G'}} f(j)$ and if $i \cancel R j$ then $V_i\neq \phi_n(V_j)$ for every $n$. Thus we may replace $G$ by $G'$ and simply assume that $G$ has uniformly computable finite orbits.
	
	We again have strategies $\mathcal{P}^n_{i,j}$ placed on a tree with outcomes $\{\infty<d<w\}$.
	
	\begin{itemize}
	\item[$\mathcal{P}_{i,j}^n$]: If $X(i,j,n)$, then make $V_j = g(V_i)$ for some $g\in G$. If $\neg X(i,j,n)$, then ensure that $V_j\neq \phi_n(V_i)$.
	\end{itemize}
		
	A $\mathcal{P}_{i,j}^n$-strategy $\alpha$, after taking the $d$-outcome, will specify a computable set $S_\alpha$ of quadruples $(a,b,k,\ell)$. 
	This will be formed by taking the union of the $S_{\beta}$ for higher priority strategies which have last take their $d$-outcomes and possibly adding some new quadruples.
		A quadruple $(a,b,k,\ell)$ in a set $S_\alpha$ is understood as saying that the possible actions sending $V_k$ to $V_\ell$ must send $a$ to $b$. A potential map $g$ from $V_x$ to $V_y$ is said to be \emph{consistent with $S_\alpha$} if for all quadruples $(a,b,x,y)\in S_\alpha$, we have $g(a)=b$.
		
		A $\mathcal{P}_{i,j}^n$-strategy will take outcome $w$ when first visited. Whenever the approximation to $X$ says that $X(i,j,n)$ holds, it will take outcome $\infty$. Otherwise, it takes either outcome $w$ or $d$, depending on whether a certain computation converges.
		
\subsection*{The strategy to meet a $P_{i,j}^n$-strategy at node $\tau$} When first visited, $\tau$ sets $S_\tau$ to be the union of $S_\beta$ for higher priority $\beta$. Then it chooses some $g^\tau \in G$ so that $g$ is a potential map from $V_i$ to $V_j$ which is consistent with this $S_\tau$. 
The strategy at $\tau$ then chooses a new number $K$ and a pair of group elements $g_0,g_1$ each potential maps from $V_i$ to $V_j$ which are consistent with $S_\tau$ so that $g_0(K)\neq g_1(K)$. In choosing $K$ to be new, we mean that $K$ is chosen to be in a different $G$-orbit than any number previously mentioned in the construction.
	Then $\tau$ restrains any element from the orbit of $K$ from entering any set $V_\ell$.
	The strategy $\tau$ will continue to take outcome $w$ until either the appoximation to $X$ says that $X(i,j,n)$, in which case it takes outcome $\infty$ or 
	we see $\phi_n(K)\downarrow$, in which case it will take outcome $d$. As long as it takes outcome $w$, it takes no further action.
	
	When taking the outcome $d$ for the first time since last taking outcome $w$, $\tau$ checks if $\phi_n(K)$ is in $V_{j,s}$. If so, it does nothing and maintains its restraint against $K$ entering $V_i$. If, on the other hand $\phi_n(K)$ is not in $V_{j,s}$, then it puts $K$ into $V_i$ and restrains $\phi_n(K)$ from entering $V_j$.
	For every single set $V_\ell$, $\tau$ then puts exactly one member of the $G$-orbit of $K$ into $V_\ell$. This is done inductively as follows: 
	\begin{enumerate}
		\item We put $g_0(K)$ into $V_j$ unless $\phi_n(K)=g_0(K)$, in which case we put $g_1(K)$ into $V_j$.
		\item For each $k$, if there is a $\beta$ so $\beta\smallfrown \infty \preceq \tau$ and $\beta$ is an $\mathcal{P}^m_{i',k}$-requirement, then the number which we put in $V_k$ is the $g^\beta$-image of the number that we put into $V_{i'}$. 
		\item For each $k$, if there is no $\beta$ as such, then we choose any $h_k\in G$ a potential map from $V_i$ to $V_k$ which is consistent with $S_\tau$ and we put $h_k(K)$ into $V_k$.
	\end{enumerate}
	Finally, we increase $S_\tau$ so that for every pair $(k,\ell)$, we put $(a,b,k,\ell)$ into $S_\tau$ where $a$ is the number we have put into $V_k$ and $b$ is the number we put into $V_\ell$.
	
	When taking the infinite outcome or if $\tau$ is injured, $\tau$ immediately places the entirety of the $G$-orbit of $K$ into every set $V_{\ell}$. If it has no parameter $K$ chosen yet, then there is no clean-up to do here, and it does nothing. It also reverts its $S_\tau$ to being the union of the $S_\beta$ for higher priority $\beta$.
		Note that $\tau$ does \emph{not} perform any action on taking outcome $\infty$ to ensure that $g^\tau(V_i)=V_j$. In lieu of this, every strategy right of $\tau$ cleans up after themselves whenever they are injured, and strategies under the outcome $\tau\smallfrown \infty$ respect $\tau$'s choice of $g^\tau$ when they put numbers into $V_j$.

	%	When taking the $f$-outcome (i.e., we believe that $\neg X(i,j,n)$) and when last visited we took the $\infty$-outcome or this is the first stage taking any outcome, we wait to find a $g$-cycle of length larger than $2L$. Let this be an element $a$ with $g$-order $N>2L$. We place $a$ into $V_i$, and we place either $g^{n^s_{i,j}}(a)$ or $g^{L+ n^s_{i,j}}(a)$ into $V_j$. In the first case, we have $n^{s+1}_{i,j} = n^s_{i,j}$ and in the latter case, we make $n^{s+1}_{i,j}=L+n^s_{i,j}$. We choose it so that $n\not\equiv n^{s+1}_{i,j} \mod N$.
	%	For every other set $V_k$, if $\tau$ is not below any $\sigma\smallfrown\infty$ for $\sigma$ a $P_{i',k}^{n'}$-strategy, then we place $g^{n^s_{i,k}}(a)$ into $V_k$. Otherwise, we see which number $b$ enters $V_{i'}$, and then place $g^{n_{i',k}^s}(a)$ into $V_k$.
	%	
	%	If we take the $\infty$-outcome, then we perform the clean-up phase.
	%	
	%	The clean-up action for $\tau$: For the number $a$ which we placed into $V_i$ when we took the $f$-outcome, we place all of the $g$-orbit of $a$ into every set $V_k$. Further, for every number $a$ which any node $\rho>_L \tau \smallfrown \infty$ placed into any $V_k$, we place all of the $g$-orbit of $a$ into each $V_k$.

	\subsection*{Verification.} The verification is based on the following lemmas.
	
	\begin{lem}
		The choice of $S_\alpha$ is consistent and coherent. That is, for every pair $k,\ell$, there is an element $g\in G$ which is a potential map from $V_k$ to $V_\ell$ which is consistent with $S_\alpha$.
		Similarly, if $\delta\smallfrown \infty \preceq \alpha$, then $g^\delta$ is consistent with $S_\alpha$.
	\end{lem}
	\begin{proof}
		We prove this by induction on stages.
		
		As every $S_\beta$ contains all the $S_\gamma$ for $\gamma$ higher priority, when $S_\alpha$ is first defined, it is set to equal some $S_\beta$ already defined, so it is consistent by inductive hypothesis. Similarly, if $\delta\smallfrown \infty \preceq \alpha$, then also $\delta\smallfrown \infty \preceq \beta$ or $\beta=\delta$, thus the second condition is also maintained when $S_\alpha$ is first defined. Thus, we only need to check that when $S_\alpha$ grows due to taking action in the $d$ outcome, the inductive hypotheses are maintained.
		
		Let $\alpha$ be a $\mathcal{P}^n_{i,j}$-strategy.
		It is immediate that coherence is maintained for the pair $i,j$, witnessed by either $g_0$ or $g_1$.
		It is immediate that coherence is maintained for all pairs $i,k$ where $k$ is in case (3), since we chose an element $h_k\in G$ which was consistent with $S_\alpha$ to decide which element of the $G$-orbit of $K$ to put into $V_k$. For all $k$ in case (2), the second inductive hypothesis shows coherence for the pair $i,k$. Namely, we have two maps $g^\beta$ and $h_k$ each consistent with $S_\alpha$, and their composition is also consistent with $S_\alpha$. Finally, composing two maps between $V_i$ and $V_k$ and $V_i$ and $V_\ell$, we see that $S_\alpha$ is coherent for every pair $k,\ell$.
		
		Finally, we chose our elements to enter $S_\alpha$ so as to be consistent with $g^\beta$ for all $\beta$ with $\beta\smallfrown \infty \preceq \alpha$ in (2).
	\end{proof}
	
	This implies that when a node $\alpha$  is first visited, it can choose its parameter $g^\alpha$.
	
	\begin{lem}\label{GdenseRestraintsMaintained}
		If a node $\alpha$ places a restraint against a number $n$ entering the set $V_i$, then either $\alpha$ is injured, lifts the restraint, or $n$ does not enter $V_i$.
	\end{lem}
	\begin{proof}
		When $\alpha$ places a restraint, it is either in outcome $w$ or $d$. Note that the restraint is lifted if it ever enters outcome $\infty$. So, the only strategies which can act, supposing that $\alpha$ is not injured and the restraint is not lifted are those to the right of $\alpha\smallfrown \infty$, which are all currently reinitialized, or nodes $\beta$ so that $\beta\smallfrown \infty \preceq \alpha$. In any case, numbers only enter sets due to clean-up or diagonalization for elements $K$ chosen after this restraint is placed. In particular, those elements are disjoint from the $G$-orbit of the restrained element, so in neither the diagonalization nor the clean-up can they cause the restrained number to enter any set.
	\end{proof}

	\begin{lem}
		Suppose $\tau$ is on the true path and is a $\mathcal{P}^n_{i,j}$-strategy. If $X(i,j,n)$, then there is an element $g\in G$ so that $g(V_i) = V_j$. If $\neg X(i,j,n)$ then $\phi_n(V_i)\neq V_j$.
	\end{lem}
	\begin{proof}
		We first consider the case where $X(i,j,n)$ holds.
		
		Consider the first stage $s_0$ at which $\tau$ is visited after its last initialization. Then it chooses an element $g^\tau \in G$. By choice of $g^\tau$ as being consistent with $S_\tau$, it is consistent with all elements which have already entered $V_i$ and $V_j$. That is, if a higher priority strategy $\beta$ has placed $x$ into $V_i$ and $y$ into $V_j$ in the same orbit, then it put the quadruple $(x,y,i,j)$ into its set $S_\beta$, and so we have $g^\tau(x)=y$.
		
		No node to the left of $\tau$ ever acts again. The cumulative future effect of nodes right of $\tau\smallfrown \infty$ or above $\tau\smallfrown \infty$ are that they place entire $G$-orbits into $V_i$ and $V_j$. 
		Finally, we have to consider which elements might enter $V_i$ and $V_j$ due to strategies below $\tau\smallfrown \infty$. These place some number $x$ into $V_i$, then, via the second bullet, they place $g^\tau(x)$ into $V_j$. Thus they also agree with $g^\tau$. Thus $V_j=g^\tau(V_i)$.
		
		Next, we consider the case where $\neg X(i,j,n)$. Consider the stage when $\tau$ takes outcome $w$ after its last initialization and after its last time taking outcome $\infty$. Then it chooses a number $K$ and places restraint that no element in the $G$-orbit of $K$ enter any set. If $\phi_n(K)$ diverges, then the requirement is satisfied, so we may assume it converges. There are two cases, and Lemma \ref{GdenseRestraintsMaintained} shows that $\phi_n(K)\in V_{j}$ if and only if $K\notin V_i$ in either case.
		%		
		%		
		%		Note that $g$ can be chosen because $f_{\sigma}$ is consistent with an element of $G$ for every $\sigma$, and $f_{<\tau}=f_\sigma$ for some $\sigma$. If it takes the infinite outcome infinitely often, then it ensures that $g(V_i) = V_j$. If the true outcome is $w$, then $\phi_n$ is not total and so $\phi_n(V_i)\neq V_j$. Suppose that the true outcome is $d$. After finitely many times of moving $K$ because it was seen to be in a restrained orbit, eventually the final choice of $K$ is made.
		%		Then there are two cases. Either $\phi_n(K)$ was not in $V_{j,s}$ and $\tau$ put $K$ into $V_i$. Due to the restraint that $\phi_n(K)$ does not enter $V_{j}$, Lemma \ref{GdenseRestraintsMaintained} shows that $\phi_n(V_i)\neq V_j$. In the other case, we maintained restraint that $K$ not enter $V_i$, so Lemma \ref{GdenseRestraintsMaintained} shows that $K\notin V_i$, but $\phi_n(K)\in V_j$ so we have $\phi_n(V_i)\neq V_j$.	
	\end{proof}
	
As above, let $f$ be so that $V_i=W_{f(i)}$ for all $i$. Then, the fact that $f$ is a reduction follows by induction as in Lemma \ref{lem:ItsAReduction}.
%	\begin{lem}
%		For every $i<j$, $i \mathrel{R} j$ if and only if $f(i) \mathrel{E_G} f(j)$.
%	\end{lem}	
%\noindent	This concludes the proof that, if $G$ has non-isolated actions, then
% $E^{ce}_G$ is $\Sigma^0_3$ complete.  
\end{proof}

\subsection{Concluding}

Recall that, by Lemma \ref{trichotomy}, the three cases considered are exhaustive and, by Lemma \ref{normal form c.e. orbit eqrels}, each c.e.\ orbit equivalent relation is of the form $R^{ce}_G$. Hence, putting  Theorems \ref{G finite}, \ref{Inf orbit}, and \ref{dense actions} together, we finally obtain the desired dichotomy:

\begin{mainthm}\label{characterization of ceoer}
Every c.e.\ orbit equivalence relation is computably bireducible with exactly one of $=^{ce}$ or $E_0^{ce}$.
%
%Up to computable reducibility, every c.e.\ orbit equivalence relation is either equivalent to $=^{ce}$ or $E_0^{ce}$.
\end{mainthm}
%\begin{mainthm}
%	For any computable group $G$ of computable permutations, $E^{ce}_G$ is either  $\Sigma^0_3$ complete (thus, equivalent to $E^{ce}_0$) or equivalent to $\Ece$.
%\end{mainthm}
%\begin{proof}
%	It follows from Lemma \ref{trichotomy} that the three cases considered in Theorems \ref{G finite}, \ref{Inf orbit}, and \ref{dense actions} are exhaustive. Thus either $E_G$ is equivalent to $\Ece$ or is $\Sigma^0_3$-complete by Theorems \ref{G finite}, \ref{Inf orbit}, \ref{dense actions}.
%\end{proof}

\section{Equivalence Relations  \eiti}

We now give examples of degrees which contain equivalence relations which are \eiti. We concentrate on the interval between $\Ece$ and $\EZce$.  

%First, note that every equivalence relation which is \eiti is $\Sigma^0_3$, thus is reducible to $\EZce$, so to show an equivalence relation is in this interval we need only show it's  above $\Ece$.
Contrasting with our dichotomy theorem for the case of a \ceoer, we show that there are both infinite chains and infinite antichains of equivalence relations which are \eiti. This gives several strong answers to \cite[Question 3.5]{CHM} showing that there is no analogue of the Glimm-Effros dichotomy for equivalence relations on the c.e.\ sets.

% The following observation is easy and \cite[Theorem 4.3]{CHM} notes that $E_{\min}$ is properly below $\Ece$. Define $i\, E_{\min} \, j \Leftrightarrow (\min W_i =\min W_j).$
%
%
%\begin{obs}
%	$E_{\min}$ is enumerable in indices. 
%\end{obs}

\begin{thm}\label{infinite chains of eqrels enumerable in indices}
	There is an infinite chain of equivalence relations which are \eiti between $\Ece$ and $\EZce$. 
\end{thm}
\begin{proof}
	For any number $i$, let $F(i)$ be the least number in $\omega\smallsetminus W_i$ and undefined if $W_i=\omega$.
	For each ceer $X$, let $R_X$ be the equivalence relation defined as follows:
\[
	i \mathrel{R_X} j \Leftrightarrow W_i=W_j \mbox{ or } (0\in W_i\cap W_j \mbox{ and } F(i)\mathrel{X} F(j)).
	\]
	
	\begin{lem}\label{TheseAreEnumerableInIndices}
		Let $X$ be  a ceer where every class is infinite. Then the relation $R_X$ is \eiti.
	\end{lem}
	\begin{proof}
		Until we see $0$ enter $W_i$, we only enumerate $i$ into the collection of indices equivalent to $i$. Once we see $0$ enter $W_i$, we take our approximation $y$ to $F(i)$ and for every $x\in [y]_X$, we enumerate all c.e.\ sets $(Y\cup [0,x-1])\smallsetminus \{x\}$ into the family. If at a later stage we see $y$ enter $W_i$, we take our new approximation $y'$ to $F(i)$ and for every old set we choose an $x'>x$ so $x' \mathrel{X} y'$ and $x'$ hasn't already been enumerated into the set and we move to enumerating $(Y\cup [0,x'-1])\smallsetminus \{x'\}$. In addition, we add a new column for each set 
		$(Y\cup [0,z-1])\smallsetminus \{z\}$ for any c.e.\ $Y$ and $z\mathrel{X} y'$.
		If we change our approximation to $F(x)$ infinitely often, each sets in our  enumerated family is $\omega$, which is exactly $W_i$. Otherwise, each set settles on something equivalent to $W_i$. Similarly, the last time we add columns, we add representatives of every c.e.\ set equivalent to $W_i$. 
	\end{proof}

	We apply the above lemma to the case of the ceer $\Id_n$ given by equivalence modulo $n$. For $n\geq 1$, let $R_n$ be the equivalence relation given as $R_{\Id_n}$. It follows that each $R_n$ is \eiti. 
	The map $k\mapsto l$ where $W_l=\{x+1 : x\in W_k\}$ reduces $\Ece$ to $R_n$ for each $n$ as these sets do not contain $0$. It suffices to show that $R_n<_c R_{n+1}$ for each $n$. The map $k\mapsto l$ where $W_l = \{(n+1)x+y : nx+y\in W_k \text{ with } 0\leq y<x\}$ gives a reduction of $R_n$ to $R_{n+1}$. Finally, observe that $R_{n+1}$ has precisely $n+1$ classes which are properly $\Sigma^0_2$. Namely, these are the classes of sets which contain $0$ but are not total. But $R_n$ only contains $n$ such classes, so it follows that $R_{n+1}\not\leq_c R_{n}$.
\end{proof}

\begin{thm}\label{infinite antichain of eqrels enumerable in indices}
	There is an infinite antichain of equivalence relations \eiti between $\Ece$ and $\EZce$.
\end{thm}
\begin{proof}
	We define the function $F(i)$ to be the least number $k$ so that $2k+1$ is not in $W_i$.
We will construct a sequence $(X_n)_{n\in \omega}$ of ceers with infinite classes. Then we define a  sequence $A_n$ of equivalence relations as follows:
	\begin{multline*}
		i \mathrel{A_n} j \Leftrightarrow W_i=W_j \text{ or (both $|W_i\cap \text{Evens}|$ and $|W_j\cap \text{Evens}|$ are $\geq 2$)} \\ 
		\text{ or (both $W_i$ and $W_j$ contained only 1 even number $2k$ and $\langle F(i),k\rangle \mathrel {X_n} \langle F(j),k\rangle$)} ).
	\end{multline*}
	
	As in Lemma \ref{TheseAreEnumerableInIndices}, each $A_n$ is \eiti. 
	In particular, whenever we believe $W_i$ contains no even numbers, we only enumerate the index $i$ as being equivalent to $i$. When we believe that $W_i$ contains exactly 1 even number, we use the enumeration from  Lemma \ref{TheseAreEnumerableInIndices}, and as soon as we see that $W_i$ contains at least two even numbers, we just enumerate indices for all c.e.\ sets containing at least two even numbers.
	
	We note that there are four kinds of indices to consider: 
	\begin{itemize}
		\item \emph{Oddish indices}: These are $i$ so that $W_i\subseteq \text{Odds}$. In this case, $[i]_{A_n} =[i]_{\Ece}$, so $[i]_{A_n}$ is properly $\Pi^0_2$. Using these indices and a function $k\mapsto \ell$ where $W_\ell = \{2x+1 : x\in W_k\}$ is a reduction of $\Ece$ to each $A_n$.
		\item \emph{Proper $k$-Coding indices}: These are $i$ so that $W_i$ contains only one even number $2k$ and $\text{Odds}\not\subseteq W_i$. Then $[i]_{A_n} = \{j: j \text{ is also a $k$-coding index and } \langle F(i),k\rangle \mathrel {X_n} \langle F(j),k\rangle)\}$. In this case $[i]_{A_n}$ is properly $\Sigma^0_2$.
		\item \emph{Full $k$-Coding indices}: These are $i$ so that $W_i$ contains only one even number $2k$ and $\text{Odds}\subseteq W_i$. In this case, $[i]_{A_n}$ is properly $\Pi^0_2$.
		\item \emph{Big indices}: These are $i$ so that $W_i$ contains $\geq 2$ even numbers. In this case $[i]_{A_n}$ is $\Sigma^0_1$.
	\end{itemize}
	
	The key idea of this argument is the following lemma:
	
	%	\begin{lem}
	%		Suppose that $f$ is a reduction of $A_n$ to $A_m$ and that $W_i$ and $W_j$ both contain exactly 1 even number $2k$. Then $W_{f(i)}$ and $W_{f(j)}$ also contain only 1 even number, and it is the same even number.
	%	\end{lem}
	%	\begin{proof}
	%		Note that the class of $i$ is properly $\Sigma^0_2$ (if $F(i)$ is defined) or is properly $\Pi^0_2$ if $F(i)$ is undefined (i.e., $\text{Odds}\subseteq W_i$).
	%		In either case, it follows that the image classes cannot be the class of sets containing $\geq 2$ even numbers (which is $\Sigma^0_1$). 	
	%	\end{proof}
	
	\begin{lem}
		Suppose that $f$ is a reduction of $A_n$ to $A_m$. Then there is a computable function $g$ so that for every $k$, if $i$ is a Proper $k$-Coding index, then $f(i)$ is a Proper $g(k)$-coding index. 
	\end{lem}
	\begin{proof}
		Since the only classes which are properly $\Sigma^0_2$ are those of proper $l$-coding indices, we see that if $i$ is a Proper $k$-Coding index, then $f(i)$ is a Proper $l$-coding index for some $l$. We need only show that $l$ depends only on $k$.
		
		Suppose towards a contradiction that $i$ and $j$ are both Proper $k$-Coding indices yet $f(i)$ is a Proper $l$-Coding index and $f(j)$ is a Proper $l'$-Coding index with $l'\neq l$. By the Recursion Theorem, we build a c.e.\ set with index $e$ as follows: We put $2k$ into $W_e$ and we wait to see $2m$ enter $W_{f(e)}$ for some $m$. If we later see $2\ell$ enter $W_{f(e)}$, we extend $W_e$ to equal $W_j$ and end the construction. On the other hand, if we see some $2\ell'$ enter $W_{f(e)}$, we extend $W_e$ to equal $W_{i}$ and end the construction.
		
		If we were to never see any $2m$ enter $W_{f(e)}$, then we have a Proper $k$-Coding index which is not sent to a Proper $l$-Coding index for any $l$, which we have already observed is impossible. In either of the other cases, we have ensured that $e\mathrel A_n i$ if and only if $f(e)\cancel{\mathrel{A_m}} f(i)$ contradicting $f$ being a reduction.
	\end{proof}
	
	We now build the ceers $X_n$ so the resulting $A_n$ satisfy the requirements:
	
\begin{itemize}
\item[$\mathcal{R}^k_{n,m}$]: $\phi_n$ is not a reduction of $A_n$ to $A_m$.
\end{itemize}
	
	We fix a ceer $E$ which has infinitely many infinite classes. We begin the construction with every $X_n$ being equal to $\oplus_{i\in \omega} E$.
	
\subsection*{The strategy to meet one $\mathcal{R}^k_{n,m}$-requirement} Suppose that the $p$th requirement in terms of priority is $\mathcal{R}^k_{n,m}$. Choose a column $\omega^{[k]}$ which has never been mentioned in the construction (in particular $X_{n,s}\restriction \omega^{[k]}=E$). We make $X_{n,s+1}\restriction \omega^{[k]}=\Id_{p+2}$, and we restrain lower priority requirements from further collapsing $X_n\restriction \omega^{[k]}$. We choose a number $K$ which is a proper $k$-coding index (i.e., we choose an index for the set $\{2k\}$). We wait to see $g(k)$ converge. That is, we wait to see $2l$ enter $W_{f(K)}$ for some $l$. Once we see this $l$, there are two cases: If $X_{m}\restriction \omega^{[l]}$ is restrained by a higher priority requirement, then $X_{m}\restriction \omega^{[l]}$ already has fewer than $p+2$ classes. Otherwise, it is not restrained and we can make 
% We need do nothing since each of the $p+2$ non-equivalent proper $k$-coding indices in $A_n$ are sent to proper $l$-coding indices in $A_m$, which cannot be injective by the pigeon-hole principle. On the other hand, if $X_{m}\restriction \omega^{[l]}$ is not restrained by a higher priority requirement, then we just make
  $X_{m}\restriction \omega^{[l]}=\Id_1$. In either case, we have satisfied the requirement since, were $\phi_n$ to be a reduction of $A_n$ to $A_m$, then each of the $p+2$ non-equivalent proper $k$-coding indices in $A_n$ would be sent to proper $l$-coding indices in $A_m$, but this cannot be injective on classes by the pigeon-hole principle (we use $p+2$ to ensure that $p+2>1$ even when $p=0$).
	In any case, whenever we act, we injure all lower priority requirements. These strategies fit together as a standard finite-injury construction.
\end{proof}

Thus unlike our dichotomy theorem in the case of a \ceoer, there is a far richer collection of degrees of equivalence relations which are \eiti, and thus there is no dichotomy theorem for equivalence relations on c.e.\ sets.

\end{document}